\documentclass[11pt]{amsart}
\usepackage{graphicx}
\usepackage{oldgerm}
\usepackage{mathdots}  
\usepackage{stmaryrd}
\usepackage{bm}
\usepackage[all]{xy}
\usepackage{color}
\usepackage{enumerate}

\usepackage{hyperref}
\usepackage{mathrsfs}
\usepackage{tikz} 
\usetikzlibrary{arrows,%
  decorations.markings,%
  matrix,%
  shapes}

\usepackage[latin1]{inputenc}
\usepackage{amsmath,amsthm, amscd, amssymb, amsfonts}

\numberwithin{equation}{section}

\theoremstyle{plain}

\numberwithin{equation}{section}

\newtheorem{theorem}[equation]{Theorem}
\newtheorem{corollary}[equation]{Corollary}
\newtheorem{proposition}[equation]{Proposition}
\newtheorem{lemma}[equation]{Lemma}

\theoremstyle{definition}
\newtheorem*{definition}{Definition}

\theoremstyle{definition}
\newtheorem{remark}[equation]{Remark}

\theoremstyle{definition}
\newtheorem*{question}{Question}

\newcommand{\A}{\mathscr{A}}
\newcommand{\C}{\mathscr{C}}

\newcommand{\D}{\mathscr{D}}

\newcommand\Aut{\operatorname{Aut}}
\newcommand\Mon{\operatorname{Mon}}
\newcommand\Ext{\operatorname{Ext}}

\newcommand\ot{\otimes}

\newcommand\Hom{\operatorname{Hom}}

\newcommand\vect{\operatorname{Vec}}

\newcommand\VC{V_{\C}}

\newcommand\unit{\mathbf{1}}

\DeclareMathOperator{\Ker}{Ker}
\newcommand{\DOT}{\setlength{\unitlength}{1pt}\begin{picture}(2.5,2)(1,1)\put(2,3){\circle*{2}}\end{picture}}
\newcommand{\bu}{\DOT}

\newcommand{\Coh}{\operatorname{H}\nolimits}
\newcommand{\Ho}{\operatorname{\Coh^{\bu}}\nolimits}
\newcommand{\Maxspec}{\operatorname{MaxSpec}\nolimits}

\newcommand{\m}{\mathfrak{m}}

\makeatletter
\def\blx@maxline{77}
\makeatother

\begin{document}
\title{Support varieties without the tensor product property}

\author{Petter Andreas Bergh, Julia Yael Plavnik, Sarah Witherspoon}

\address{Petter Andreas Bergh \\ Institutt for matematiske fag \\
NTNU \\ N-7491 Trondheim \\ Norway} \email{petter.bergh@ntnu.no}
\address{Julia Yael Plavnik \\ Department of Mathematics \\ Indiana University \\ Bloomington \\ Indiana 47405 \\ USA\\ \& Fachbereich Mathematik\\ Universit\"at Hamburg\\ Hamburg  20146\\ Germany}
\email{jplavnik@iu.edu}
\address{Sarah Witherspoon \\ Department of Mathematics \\ Texas A\&M University \\ College Station \\ Texas 77843 \\ USA}
\email{sjw@math.tamu.edu}
\subjclass[2020]{16E40, 16T05, 18M05, 18M15}
\keywords{Finite tensor categories; support varieties; tensor product property}

\begin{abstract}
We show that over a perfect field, every non-semisimple finite tensor category with finitely generated cohomology embeds into a larger such category where the tensor product property does not hold for support varieties.
\end{abstract}

\maketitle

\section{Introduction}

In the two recent papers \cite{BPW1, BPW2}, we studied support varieties in the setting of finite tensor categories. When the cohomology of such a category is finitely generated---as conjectured by Etingof and Ostrik to be always true---then the varieties contain much homological information on the objects, and the theory resembles that for support varieties over group algebras and more general cocommutative Hopf algebras.

In \cite{BPW2}, we focused on the \emph{tensor product property} for support varieties. That is, given a finite tensor category $\C$ with finitely generated cohomology, we studied conditions under which the equality 
$$\VC ( X \ot Y ) = \VC (X) \cap \VC (Y)$$
holds for all objects $X,Y \in \C$. This property is crucial if one for example wants to use support varieties to classify the thick tensor ideals in the stable category. It is well known that there are non-braided finite tensor categories where the property does not hold, see for example \cite{BW, PW}. However, we showed in \cite{BPW2} that when the category is braided, the tensor product property holds for all objects if and only if it holds between indecomposable periodic objects.

In this paper, we show that when the ground field is perfect, then every non-semisimple finite tensor category $\C$ with finitely generated cohomology embeds into one such category $\D$ where the tensor product property does not hold. This is true even if the tensor product property \emph{does} hold in $\C$. The category $\D$ that we construct is a crossed product category that is not braided;
along the way we collect facts about such crossed product categories that may be of
independent interest.  It remains an open question whether the tensor product property always holds in the braided case.

\subsection*{Acknowledgments}
P.A.\ Bergh would like to thank the organizers of the Representation Theory program hosted by the Centre for Advanced Study at The Norwegian Academy of Science and Letters, where he spent parts of fall 2022. J.Y.\ Plavnik was partially supported by NSF grant DMS-2146392 and by Simons Foundation Award 889000 as part of the Simons Collaboration on Global Categorical Symmetries. J.Y.P.\ would like to thank the hospitality
and excellent working conditions at the Department of Mathematics at Universit\"at Hamburg, where she has carried out part of this research as an Experienced Fellow of the Alexander von Humboldt Foundation. S.J.\ Witherspoon was partially supported by NSF grant 2001163.

\section{Preliminaries}\label{sec:prelim}

We fix a field $k$ that is not necessarily algebraically closed, together with a finite tensor $k$-category $( \C, \ot, \unit )$ in the sense of \cite{EGNO}. This means that $\C$ is a locally finite $k$-linear abelian category, with a finite set of isomorphism classes of simple objects. Moreover, every object admits a projective cover, and hence also a minimal projective resolution. Furthermore, there is a bifunctor $\ot$ from $\C \times \C$ to $\C$, associative up to functorial isomorphisms, and called the tensor product. There is also a unit object $\unit$ with respect to the tensor product, and $( \C, \ot, \unit )$ is a monoidal category. In particular, the tensor product satisfies the so-called pentagon axiom; see \cite[Section 2.1]{EGNO}. The unit object is simple, and the monoidal structure is compatible with the abelian structure in that the tensor product is bilinear on morphisms. Finally, every object admits a left and a right dual in the sense of \cite[Section 2.10]{EGNO}, so that $\C$ is rigid as a monoidal category.

The rigidity of $\C$ has important consequences; we mention three of them. First of all, by \cite[Proposition 4.2.1]{EGNO}, the tensor product is biexact, that is, exact in each argument. Secondly, by \cite[Proposition 4.2.12]{EGNO}, the projective objects form a two-sided ideal in $\C$, so that the tensor product between a projective object and any other object is again projective. Finally, by \cite[Proposition 6.1.3]{EGNO}, the projective and the injective objects of $\C$ are the same, so that the category is actually quasi-Frobenius.

Given objects $M,N \in \C$, we denote the graded $k$-vector space $\oplus_{n=0}^{\infty} \Ext_{\C}^n (M,N)$ by $\Ext_{\C}^* (M,N)$. With the usual Yoneda product as multiplication, the space $\Ext_{\C}^* (M,M)$ becomes a graded $k$-algebra, and of particular interest is the algebra $\Ext_{\C}^* ( \unit, \unit)$. This is the \emph{cohomology ring} of $\C$, and denoted by $\Coh^* ( \C )$. By \cite[Theorem 1.7]{SA}, this is a graded-commutative $k$-algebra. Since the tensor product is exact in the first argument, the functor $- \ot M$ induces a homomorphism
$$\Coh^* ( \C ) \xrightarrow{\varphi_M} \Ext_{\C}^* (M,M)$$
of graded $k$-algebras, turning $\Ext_{\C}^* (M,M)$ into a left and a right $\Coh^* ( \C )$-module. Now since $\Ext_{\C}^* (M,N)$ is a left $\Ext_{\C}^* (N,N)$-module and a right $\Ext_{\C}^* (M,M)$-module (again using the Yoneda product), we see that it is both a left and a right module over $\Coh^* ( \C )$, via $\varphi_N$ and $\varphi_M$, respectively. However, by \cite[Lemma 2.2]{BPW2} the two module actions coincide for homogeneous elements, up to a sign. In particular, it makes no difference whether we view $\Ext_{\C}^* (M,M)$ as a left or as a right module over $\Coh^* ( \C )$.

Since the cohomology ring is graded-commutative, the graded $k$-algebra defined by
$$\Ho ( \C ) = \left \{ 
\begin{array}{ll}
\Coh^*( \C ) & \text{if the characteristic of $k$ is two,} \\
\Coh^{2*}( \C ) & \text{if not}
\end{array} 
\right.$$
is commutative in the ordinary sense. We denote by $\m_0$ the ideal $\Coh^+ ( \C )$ of this ring, that is, the ideal of $\Ho ( \C )$ generated by the homogeneous elements of positive degree. This is a maximal ideal, since $\Coh^0 ( \C ) = \Hom_{\C}( \unit, \unit )$ is a field; it is a division ring since the unit object is simple, and commutative by the above discussion. 

\begin{definition}
The \emph{support variety} of an object $M \in \C$ is
$$\VC (M) = \{ \m_0 \} \cup \{ \m \in \Maxspec \Ho ( \C ) \mid \Ker \varphi_M \subseteq \m \}$$
\end{definition}

Note that the presence of $\m_0$ in the definition of support varieties is superfluous whenever $M$ is nonzero, for then this maximal ideal automatically contains the homogeneous ideal $\Ker \varphi_M$. Without any finiteness condition on the cohomology of $\C$, these support varieties may not contain any important homological information, and so we make the following definition.

\begin{definition}
The finite tensor category $\C$ satisfies the \emph{finiteness condition} \textbf{Fg} if the cohomology ring $\Coh^*( \C )$ is finitely generated, and $\Ext_{\C}^*(M,M)$ is a finitely generated $\Coh^*( \C )$-module for every object $M \in \C$.
\end{definition}

By \cite[Remark 3.5]{BPW1}, one can replace $\Coh^*( \C )$ by $\Ho ( \C )$ in this definition; the two versions are equivalent. It was conjectured by Etingof and Ostrik in \cite{EO} that \emph{every} finite tensor category satisfies \textbf{Fg}, and this conjecture is still open. As shown in \cite{BPW1}, when this finiteness condition holds, then the theory of support varieties becomes quite powerful, as in the classical case for modules over group algebras of finite groups.

In this paper, we are concerned with the question of whether support varieties respect tensor products, in the following sense.

\begin{definition}
The finite tensor category $\C$ satisfies the \emph{tensor product property} for support varieties if $\VC ( M \ot N ) = \VC (M) \cap \VC (N)$ for all objects $M,N \in \C$.
\end{definition}

This definition makes perfect sense without assuming that $\C$ satisfies \textbf{Fg}. By \cite[Proposition 3.3(v)]{BPW1}, the inclusion $\VC ( M \ot N ) \subseteq \VC (M) \cap \VC (N)$ always holds when $\C$ is \emph{braided}, that is, when for all objects $M,N \in \C$ there are functorial isomorphisms $b_{M,N} \colon M \ot N \longrightarrow N \ot M$ that satisfy the hexagonal identities in \cite[Definition 8.1.1]{EGNO}. In \cite{BW} and \cite{PW}, examples are given of finite tensor categories where the tensor product property does not hold, in fact not even the above inclusion. These examples are then necessarily non-braided. It is an open question whether the tensor product property always holds in the braided case, or under the stronger requirement that $\C$ is \emph{symmetric}, that is, when the braiding isomorphisms satisfy $b_{N,M} \circ b_{M,N} = 1_{M \ot N}$ for all $M,N \in \C$.
Other than categories of modules of some types of Hopf algebras, 
the only case that has been completely settled is when the ground field is algebraically closed and of characteristic zero; over such a field, every symmetric finite tensor category satisfies the tensor product property, by \cite[Theorem 4.9]{BPW2}. The proof provided relies on Deligne's classification of such categories as certain skew group algebras, from \cite{Deligne2}. 

By \cite[Theorem 3.6]{BPW2}, when $\C$ is braided and satisfies \textbf{Fg}, the tensor product property holds if and only if the following holds for all $M,N \in \C$: if $\VC(M) \cap \VC(N)$ is not trivial, that is, if $\VC(M) \cap \VC(N) \neq \{ \m_0 \}$, then $M \ot N$ is not projective. Consequently, if the tensor product property does not hold, then there must exist two nonprojective objects $M,N$ whose tensor product $M \ot N$ is projective, but for which $\VC(M) \cap \VC(N)$ is not trivial. They must be nonprojective since the variety of a projective object is necessarily trivial; see the paragraph following \cite[Definition 3.1]{BPW1}. In the following result, we show that at least such a pair of objects with $M = N$ cannot exist.

\begin{proposition}\label{prop:powers}
Let $k$ be a field and $( \C, \ot, \unit )$ a braided finite tensor $k$-category. Then an object $M \in \C$ is projective if and only if the $n$-fold tensor product $M^{\ot n}$ is projective for some $n \ge 1$.
\end{proposition}

\begin{proof}
If $M$ is projective, then so is every tensor product $M^{\ot n}$, since the projective objects form an ideal in $\C$. Conversely, suppose that $M^{\ot n}$ is projective for some $n \ge 2$. Since $\C$ is rigid, the object $M$ admits a left dual $M^*$, which implies that there exist morphisms 
$$M \longrightarrow M \ot M^* \ot M \longrightarrow M$$
whose composition equals the identity on $M$; see \cite[Definition 2.10.1]{EGNO}. Tensoring with $M^{\ot (n-2)}$, and using the fact that $\C$ is braided, we obtain morphisms
 $$M^{\ot (n-1)} \longrightarrow M^{\ot n} \ot M^* \longrightarrow M^{\ot (n-1)}$$
whose composition equals the identity on $M^{\ot (n-1)}$. This implies that $M^{\ot (n-1)}$ is a direct summand of $M^{\ot n} \ot M^*$, which is a projective object since $M^{\ot n}$ is. Consequently, the object $M^{\ot (n-1)}$ is also projective. Repeating the process, we eventually end up with $M$, which must then be projective.
\end{proof}

Let us now in the last part of this section recall a construction that will play an important role in the main result. Suppose that $( \C, \ot_{\C}, \unit_{\C} )$ and $( \D, \ot_{\D}, \unit_{\D} )$ are two finite tensor $k$-categories. Their \emph{Deligne tensor product}, denoted $\C \boxtimes \D$, is a $k$-linear abelian category that is universal with respect to right exact bifunctors on $\C \times \D$. In other words, there is a bifunctor $T \colon \C \times \D \longrightarrow \C \boxtimes \D$ of $k$-linear abelian categories, right exact in both variables, with the property that for every bifunctor $F \colon \C \times \D \longrightarrow \A$ of $k$-linear abelian categories, the following hold: if $F$ is also right exact in both variables, then there exists a unique right exact functor $F' \colon \C \boxtimes \D \longrightarrow \A$ of $k$-linear abelian categories, with the property that the diagram
$$\xymatrix{
\C \times \D \ar[rr]^T \ar[rd]^F && \C \boxtimes \D \ar[dl]_{F'} \\
& \A }$$
commutes. The Deligne tensor product was introduced in \cite{Deligne1}; it exists, is unique up to equivalence, and is again a finite tensor category. Moreover, the bifunctor $T$ is actually exact in both variables; for details, we refer to \cite[Sections 1.11 and 4.6]{EGNO}. 

Given objects $C \in \C$ and $D \in \D$, it is standard to denote the image in $\C \boxtimes \D$ of the object $(C,D) \in \C \times \D$ by $C \boxtimes D$. When we restrict the tensor product in $\C \boxtimes \D$ to such objects, we are basically using the original tensor products. Thus if $C,C' \in \C$ and $D,D' \in \D$, then
$$(C \boxtimes D) \ot (C' \boxtimes D') = (C  \ot_{\C} C') \boxtimes (D  \ot_{\D} D')$$
where $\ot$ denotes the tensor product in $\C \boxtimes \D$. The unit object in $\C \boxtimes \D$ is $ \unit_{\C} \boxtimes  \unit_{\D}$. Moreover, there is an isomorphism
$$\Hom_{\C \boxtimes \D}(C \boxtimes D, C' \boxtimes D') \simeq \Hom_{\C}(C,C') \ot_k \Hom_{\D}(D,D')$$ 
of vector spaces, and using this, one can show that
$$\Ext_{ \C \boxtimes \D }^* ( C \boxtimes D, C \boxtimes D ) \simeq \Ext_{\C}^* (C,C) \ot_k \Ext_{\D}^* (D,D)$$
as graded $k$-algebras for $C \in \C$ and $D \in \D$. In particular, there is an isomorphism
$$\Coh^*( \C \boxtimes \D ) \simeq \Coh^*( \C ) \ot_k \Coh^*( \D )$$
of cohomology rings. Therefore, if the categories $\C$ and $\D$ both satisfy \textbf{Fg}, then we see immediately that $\Coh^*( \C \boxtimes \D )$ is finitely generated, so that at least half of \textbf{Fg} also holds for $\C \boxtimes \D$. However, if the ground field $k$ is perfect, then by \cite[Lemma 5.3]{NP} the Deligne tensor product satisfies \textbf{Fg} if and only if it holds for both $\C$ and $\D$. Moreover, in this situation, the Krull dimension of $\Coh^*( \C \boxtimes \D )$ is the sum of the Krull dimensions of $\Coh^*( \C )$ and $\Coh^*( \D )$.

\section{The main result}\label{sec:main}

In this main section, we show that every finite tensor category that satisfies \textbf{Fg} embeds into a finite tensor category that also satisfies \textbf{Fg}, but for which the tensor product property does not hold. The construction of the bigger category uses the Deligne tensor product, as well as the notion of crossed product categories that we recall next. As before, we fix a field $k$ and a finite tensor $k$-category $( \C, \ot, \unit )$.

Suppose that a finite group $G$ acts on $\C$ by tensor autoequivalences. This means that there exists a monoidal functor $\Mon (G) \longrightarrow \Aut_{\otimes} ( \C )$, where $\Aut_{\otimes} ( \C )$ is the monoidal category of tensor autoequivalences on $\C$, and $\Mon (G)$ is the monoidal category whose objects are the elements of $G$, the only morphisms are the identity maps, and the monoidal product is the multiplication in $G$. For an element $\alpha \in G$, we denote by $\alpha_*$ the corresponding tensor autoequivalence on $\C$, so that the action of $\alpha$ on an object $M \in \C$ is $\alpha_*(M)$. Note that if $\beta \in G$ is another element, then by definition there is a coherent isomorphism $( \alpha \beta )_* \simeq \alpha_* \circ \beta_*$ in $\Aut _{\otimes}( \C )$.

Following \cite{Tambara} and \cite{Nikshych}, when $G$ acts on $\C$ as above, we define the \emph{crossed product category} $\C \rtimes G$ as follows. As a $k$-linear abelian category, it is $G$-graded, and equal to $\C$ in each degree. Thus the objects in $\C \rtimes G$ are of the form $\oplus_{\alpha \in G} (M_{\alpha}, \alpha )$, with $M_{\alpha}$ an object in $\C$ for each $\alpha \in G$, and a morphism from $\oplus_{\alpha \in G} (M_{\alpha}, \alpha )$ to $\oplus_{\alpha \in G} (N_{\alpha}, \alpha )$ is a sum $\oplus_{\alpha \in G} (f_{\alpha}, \alpha )$, where $f_{\alpha} \colon M_{\alpha} \longrightarrow N_{\alpha}$ is a morphism in $\C$. We define the tensor product on homogeneous objects and morphisms by
\begin{eqnarray*}
(M, \alpha ) \ot (N, \beta ) & = & (M \ot \alpha_*(N), \alpha \beta ), \\
(f, \alpha ) \ot (g, \beta ) & = & (f \ot \alpha_*(g), \alpha \beta ).
\end{eqnarray*}
In this way, the crossed product category becomes a $G$-graded finite tensor category, with unit object $( \unit, e )$, where $e$ is the identity element of $G$. The construction is in some sense a categorification of skew group algebras. Note that $\C$ embeds as a finite tensor category into $\C \rtimes G$, via the assignment $M \mapsto (M,e)$, for $M \in \C$.

As an abelian category, the crossed product category is a Deligne product. Namely, let $\vect_G$ be the category of $G$-graded finite dimensional vector spaces over $k$, and consider the functor $T \colon \C \times \vect_G \longrightarrow \C \rtimes G$ defined as follows. The image of an object $( M, V )$ is $\oplus_{\alpha \in G} (M^{\dim V_{\alpha}}, \alpha )$, where $M^n$ denotes the direct sum of $n$ copies of $M$. Given a morphism $M \longrightarrow N$ in $\C$, the image of the corresponding morphism $(M, V) \longrightarrow (N,V)$ is the obvious morphism from $\oplus_{\alpha \in G} (M^{\dim V_{\alpha}}, \alpha )$ to $\oplus_{\alpha \in G} (N^{\dim V_{\alpha}}, \alpha )$. Finally, suppose that $\psi \colon V \longrightarrow W$ is a morphism in $\vect_G$, that is, a tuple $( \psi_{\alpha} )_{\alpha \in G}$ with each $\psi_{\alpha} \colon V_{\alpha} \longrightarrow W_{\alpha}$ a linear transformation. Fixing bases for $V_{\alpha}$ and $W_{\alpha}$, we may view $\psi_{\alpha}$ as a matrix $(c_{ij})$ with each $c_{ij} \in k$, from which we obtain a corresponding morphism $M^{\dim V_{\alpha}} \longrightarrow M^{\dim W_{\alpha}}$ in $\C$ given by the matrix $(c_{ij}1_M)$. One now checks that $T$ is a well defined bifunctor of $k$-linear abelian categories, and right exact in each variable. Moreover, given any $k$-linear abelian category $\A$ together with a $k$-linear bifunctor $F \colon \C \times \vect_G \longrightarrow \A$ which is right exact in each variable, we can construct a right exact functor $F' \colon \C \rtimes G \longrightarrow \A$ as follows. Given an object $\oplus_{\alpha \in G} (M_{\alpha}, \alpha )$ in $\C \rtimes G$, let $V$ be the $G$-graded vector space which is just $k$ in each degree, and define
$$F' \left ( \oplus_{\alpha \in G} (M_{\alpha}, \alpha ) \right ) = F \left ( \oplus_{\alpha \in G} M_{\alpha}, V \right ).$$
A morphism 
$$\oplus_{\alpha \in G} (f_{\alpha}, \alpha ) \colon \oplus_{\alpha \in G} (M_{\alpha}, \alpha ) \longrightarrow \oplus_{\alpha \in G} (N_{\alpha}, \alpha )$$
in $\C \rtimes G$ induces a morphism between $( \oplus_{\alpha \in G} M_{\alpha}, V )$ and $( \oplus_{\alpha \in G} N_{\alpha}, V )$ in $\C \times \vect_G$, and we define $F' ( \oplus_{\alpha \in G} (f_{\alpha}, \alpha ) )$ to be the image under $F$ of the latter. One now checks that $F'$ is a well defined functor of $k$-linear abelian categories, and that the diagram
$$\xymatrix{
\C \times \vect_G \ar[rr]^T \ar[rd]^F && \C \rtimes G \ar[dl]_{F'} \\
& \A }$$
commutes. Furthermore, one checks that $F'$ is unique with this property. This shows that $\C \rtimes G$ is the Deligne product $\C \boxtimes \vect_G$ as an abelian category but \emph{not} as a finite tensor category when we view $\vect_G$ as a fusion category. After all, the monoidal structure in $\C \boxtimes \vect_G$ does not use the categorical $G$-action on $\C$.

Since $\C \rtimes G = \C \boxtimes \vect_G$ as a $k$-linear abelian category, the cohomology ring $\Coh^*( \C \rtimes G )$ is isomorphic to the tensor product $\Coh^*( \C ) \ot_k \Coh^*( \vect_G )$; this does not use the monoidal structures in the categories involved. Now as $\vect_G$ is a fusion category, its cohomology ring is trivial, and so $\Coh^*( \C \rtimes G ) \simeq \Coh^*( \C )$. Consequently, when \textbf{Fg} holds for either $\C$ or $\C \rtimes G$, then at least the cohomology ring of the other category is finitely generated. However, the following lemma shows that \textbf{Fg} holds for one of the categories if and only if it holds for the other. Moreover, the support varieties for the objects of $\C \rtimes G$ are just unions of support varieties over $\C$.

\begin{lemma}\label{lem:crossedproductfg}
Let $k$ be a field, $( \C, \ot, \unit )$ a finite tensor $k$-category with a categorical action from a finite group $G$, and $\C \rtimes G$ the corresponding crossed product category. Then the following hold.

\emph{(1)} There is an isomorphism $\Coh^*( \C \rtimes G ) \simeq \Coh^*( \C )$ of cohomology rings.

\emph{(2)} $\C$ satisfies \emph{\textbf{Fg}} if and only if $\C \rtimes G$ does.

\emph{(3)} If $\oplus_{\alpha \in G} (M_{\alpha}, \alpha )$ is an object in $\C \rtimes G$, then 
$$V_{\C \rtimes G} \left ( \oplus_{\alpha \in G} (M_{\alpha}, \alpha ) \right ) = \bigcup_{\alpha \in G} \VC ( M_{\alpha} )$$
when we use the isomorphism from \emph{(1)} to replace $\Ho ( \C \rtimes G )$ by $\Ho ( \C )$.
\end{lemma}

\begin{proof}
We saw an argument for (1) above, but we now give an elementary argument for both (1) and (2). Namely, since the morphisms in $\C \rtimes G$ respect the $G$-grading, the cohomology of $\C \rtimes G$ takes place in each individual degree. The projective objects are of the form $\oplus_{\alpha \in G} (P_{\alpha}, \alpha )$, with each $P_{\alpha}$ projective in $\C$, and a (minimal) projective resolution of an object $\oplus_{\alpha \in G} (M_{\alpha}, \alpha )$ is of the form $\oplus_{\alpha \in G} ( \mathbf{P}_{\alpha}, \alpha )$, with each $\mathbf{P}_{\alpha}$ a (minimal) projective resolution of $M_{\alpha}$. Therefore, given another object $\oplus_{\alpha \in G} (N_{\alpha}, \alpha )$, there is a natural isomorphism
\begin{equation*}\label{eq:iso}
\Ext_{\C \rtimes G}^* \left ( \oplus_{\alpha \in G} (M_{\alpha}, \alpha ), \oplus_{\alpha \in G} (N_{\alpha}, \alpha ) \right ) \simeq \bigoplus_{\alpha \in G} \Ext_{\C}^* \left ( M_{\alpha}, N_{\alpha} \right ) \tag{$\dagger$},
\end{equation*}
which is an isomorphism of rings when $\oplus_{\alpha \in G} (M_{\alpha}, \alpha ) = \oplus_{\alpha \in G} (N_{\alpha}, \alpha )$. Note that since the unit object in $\C \rtimes G$ is $( \unit, e)$, it follows immediately that $\Coh^*( \C \rtimes G ) \simeq \Coh^*( \C )$, proving (1).

Suppose that $\C$ satisfies \textbf{Fg}. Then by (1) the cohomology ring $\Coh^*( \C \rtimes G )$ is finitely generated. If $X = \oplus_{\alpha \in G} (M_{\alpha}, \alpha )$ is an object of $\C \rtimes G$, then using the above isomorphism (\ref{eq:iso}), we see that the cohomology ring $\Coh^*( \C \rtimes G )$ acts on $\Ext_{\C \rtimes G}^*(X,X)$ in a way that respects the $G$-grading. That is, the action is induced by the action of $\Coh^*( \C)$ on each $\Ext_{\C}^* ( M_{\alpha}, M_{\alpha})$. Since the latter is a finitely generated $\Coh^*( \C)$-module for each $\alpha \in G$, we see that $\Ext_{\C \rtimes G}^*(X,X)$ is finitely generated as a module over $\Coh^*( \C \rtimes G )$, and so $\C \rtimes G$ satisfies \textbf{Fg}. Conversely, if the crossed product category satisfies \textbf{Fg}, then $\Coh^*( \C )$ is finitely generated by (1) again. Moreover, if $M$ is an object of $\C$, then $\Ext_{\C \rtimes G}^*( (M,e),(M,e))$ is a finitely generated $\Coh^*( \C \rtimes G )$-module. Using the isomorphism (\ref{eq:iso}), we then see that $\Ext_{\C}^* ( M, M)$ is finitely generated as a module over $\Coh^*( \C )$, so that $\C$ satisfies \textbf{Fg}. This proves (2).

For (3), we use again that the cohomology of $\C \rtimes G$ respects the $G$-grading. Given an object $(M, \alpha ) \in \C \rtimes G$ concentrated in degree $\alpha$, consider the composition
$$\Coh^*( \C ) \longrightarrow \Coh^*( \C \rtimes G ) \xrightarrow{- \ot (M, \alpha )} \Ext_{\C \rtimes G}^*( (M, \alpha ),(M, \alpha )) \longrightarrow \Ext_{\C}^* ( M, M)$$
of graded ring homomorphisms, where the outer ones are the isomorphisms from (\ref{eq:iso}). The composition equals $\varphi_M$, that is, the homomorphism $- \ot M$. Thus when we compute support varieties by using $\Ho ( \C )$, we see that $V_{\C \rtimes G} ( (M, \alpha )) = \VC (M)$. For an arbitrary object $\oplus_{\alpha \in G} (M_{\alpha}, \alpha )$ of $\C \rtimes G$ ,we then see that
$$V_{\C \rtimes G} \left ( \oplus_{\alpha \in G} (M_{\alpha}, \alpha ) \right ) = \bigcup_{\alpha \in G} V_{\C \rtimes G} \left ( (M, \alpha ) \right ) = \bigcup_{\alpha \in G} \VC ( M_{\alpha} ),$$
since support varieties respect direct sums by \cite[Proposition 3.3(i)]{BPW1}.
\end{proof}

The group $G$ acts on the crossed product category $\C \rtimes G$ by tensor autoequivalences in a natural way. Namely, for an element $\alpha \in G$, the action on objects and morphisms in $\C \rtimes G$ is given by
\begin{eqnarray*}
  \alpha_* \left ( \oplus_{\beta \in G} (M_{\beta}, \beta ) \right ) & = & \oplus_{\beta \in G} \left ( \alpha_*(M_{\beta}) , \alpha\beta\alpha^{-1} \right ) ,\\
   \alpha_* \left ( \oplus_{\beta \in G} (f_{\beta}, \beta ) \right ) & = & \oplus_{\beta \in G} \left ( \alpha_*(f_{\beta}) , \alpha\beta\alpha^{-1} \right ),
\end{eqnarray*}
where we have used the notation $\alpha_*$ to denote the tensor autoequivalences on both $\C$ and $\C \rtimes G$. The following result shows that when the tensor product property holds for $\C$, then a twisted version holds for the crossed product category.

\begin{proposition}
Let $k$ be a field, and $(\C, \ot, \unit )$ a non-semisimple finite tensor $k$-category that satisfies the tensor product property for support varieties. Furthermore, let $G$ be  finite group acting on $\C$ by tensor autoequivalences. Then for any objects $(M,\alpha)$ and $(N,\beta)$ of $\C\rtimes G$, concentrated in degrees $\alpha$ and $\beta$, the following holds:
$$V_{\C \rtimes G} \left ( (M,\alpha) \ot (N,\beta) \right ) = V_{\C\rtimes G}	\left ( (M,\alpha) \right ) \cap V_{\C\rtimes G} \left ( \alpha_*(N, \beta ) \right )$$
  \end{proposition}

\begin{proof}
By the definition of the tensor product in $\C\rtimes G$ and Lemma~\ref{lem:crossedproductfg}(3), we have 
\begin{eqnarray*}
V_{\C \rtimes G} \left ( (M,\alpha) \ot (N,\beta) \right ) & = & V_{\C\rtimes G} \left ( M \ot \alpha_*(N), \alpha \beta \right ) \\
&=& V_{\C} \left (M \ot \alpha_*(N) \right ) \\
&=& V_{\C} (M) \cap V_{\C} (\alpha_*(N))\\
& = & V_{\C\rtimes G} \left ( (M,\alpha) \right ) \cap V_{\C\rtimes G} \left ( \alpha_*(N), \alpha \beta \alpha^{-1} \right ) \\
& = & V_{\C\rtimes G}	\left ( (M,\alpha) \right ) \cap V_{\C\rtimes G} \left ( \alpha_*(N, \beta ) \right )
\end{eqnarray*}
\end{proof}  

In general, it is not always the case that $V_{\C\rtimes G}(\alpha_*(N, \beta ))$ is equal to $V_{\C\rtimes G} (N,\beta)$, or equivalently (by Lemma \ref{lem:crossedproductfg}(3)), that $V_{\C}(N)$ is equal to $V_{\C}( \alpha_*(N))$. Therefore, the above proposition may be used to construct examples where the tensor product property does not hold. However, it turns out that it is in fact not necessary to assume that the tensor product property holds for $\C$ to construct such examples. Inspired by the twisted version of the tensor product property given in the proposition, we formalize such a class of examples in a larger context next. Specifically, we will combine the Deligne tensor product with a crossed product of a specific kind. As we shall see, when the finite tensor category $\C$ that we start with is not semisimple (that is, not a fusion category), then the finite tensor category that we construct turns out not to satisfy the tensor product property. 

Let $C_2 = \{ e, \alpha \}$ be the multiplicative group with two elements, where $e$ is the identity. Consider the twisting map $\tau \colon \C \times \C \longrightarrow \C \times \C$ given by interchanging the factors, that is, mapping an object $(M,N)$ to $(N,M)$, and similarly for morphisms. This is a bilinear functor, and exact in each variable. Composing with the biexact structure bifunctor $T \colon \C \times \C \to \C \boxtimes \C$, we use the universal property of the Deligne tensor product to obtain a unique right exact functor $\alpha_* \colon \C \boxtimes \C \longrightarrow \C \boxtimes \C$ making the diagram
$$\xymatrix{
\C \times \C \ar[rr]^{T} \ar[rd]^{T \circ \tau} && \C \boxtimes \C \ar[dl]_{\alpha_*} \\
&  \C \boxtimes \C }$$
commute. The functors $T$ and $\tau$ are monoidal, hence so is $\alpha_*$, making it a functor of finite tensor categories. Moreover, from the diagram above we obtain
$$\alpha_* \circ \alpha_* \circ T = \alpha_* \circ T \circ \tau = T \circ \tau \circ \tau = T$$
and from the universal property of $T$ we may conclude that $\alpha_* \circ \alpha_*$ is the identity. Thus $\alpha_*$ is an autoequivalence of order two, and there is a monoidal functor 
$$\Mon ( C_2 ) \longrightarrow \Aut_{\otimes} (  \C \boxtimes \C )$$ 
mapping $\alpha$ to $\alpha_*$. We shall say that $C_2$ acts on $\C \boxtimes \C$ by interchanging factors, since for objects $C,C' \in \C$ there is an equality
$$\alpha_* ( C \boxtimes C' ) = \alpha_* \circ T \left ( (C,C' ) \right ) =  T \left ( (C',C ) \right ) = C' \boxtimes C$$

We may now form the crossed product category $( \C \boxtimes \C )  \rtimes C_2$. When the ground field $k$ is perfect and $\C$ satisfies \textbf{Fg}, then as mentioned in Section \ref{sec:prelim}, the Deligne tensor product $\C \boxtimes \C$ also satisfies \textbf{Fg}, by \cite[Lemma 5.3]{NP}. Then in turn so does $( \C \boxtimes \C )  \rtimes C_2$, by Lemma \ref{lem:crossedproductfg}(2). The following theorem, our main result, shows that if $\C$ is not a fusion category, that is, not semisimple, then $( \C \boxtimes \C )  \rtimes C_2$ does not satisfy the tensor product property for support varieties.

\begin{theorem}\label{thm:main}
Let $k$ be a perfect field and $( \C, \ot, \unit )$ a non-semisimple finite tensor $k$-category that satisfies \emph{\textbf{Fg}}. Furthermore, let $C_2$ be the multiplicative group of order two, acting on $\C \boxtimes \C$ by interchanging factors. Then the finite tensor $k$-category $( \C \boxtimes \C )  \rtimes C_2$ satisfies \emph{\textbf{Fg}}, but not the tensor product property for support varieties.
\end{theorem}

\begin{proof}
For simplicity, we denote the crossed product category $( \C \boxtimes \C )  \rtimes C_2$ by $\D$. In the course of the proof, we shall be using the tensor products in all the three categories $\C$, $\C \boxtimes \C$, and $\D$. To distinguish them, we therefore denote them by $\ot$, $\ot_1$, and $\ot_2$, respectively.

We saw in the paragraph preceding the theorem that $\D$ satisfies \textbf{Fg}. Now, since $\C$ is not semisimple, we may choose a nonprojective object $M \in \C$, for example the unit object; if $\unit$ were projective, then so would be every object $N \in \C$, since $N \simeq N \ot \unit$ and the projectives form an ideal. Choose a projective object $P \in \C$ for which there exists an epimorphism $P \longrightarrow M$; there exists such an object since $\C$ has enough projectives. Note that $P$ is nonzero since $M$ is not projective. Let us denote the object $( P \boxtimes M, \alpha )$ of $\D$ by just $X$, where $\alpha$ is the element of $C_2$ of order two. We shall show that
$$V_{\D} ( X \ot_2 X ) \neq V_{\D} ( X )$$
and consequently that the tensor product property for support varieties in $\D$ does not hold, since trivially $V_{\D} ( X ) \cap V_{\D} ( X ) = V_{\D} ( X )$.

By \cite[Corollary 4.2]{BPW1}, since $\C$ satisfies \textbf{Fg} and $M$ is not projective, the support variety $V_{\C}(M)$ is not trivial. Then by \cite[Proposition 6.2]{BPW1}, the $k$-vector space $\Ext_{\C}^n(M,M)$ is nonzero for infinitely many $n \ge 1$. Consider now the object $P \boxtimes M$ of $\C \boxtimes \C$. At the end of Section \ref{sec:prelim}, we saw that there is an isomorphism
$$\Ext_{ \C \boxtimes \C }^* ( P \boxtimes M, P \boxtimes M ) \simeq \Ext_{\C}^* (P,P) \ot_k \Ext_{\C}^* (M,M)$$
of $k$-vector spaces, and so since $P$ is nonzero we see that $\Ext_{ \C \boxtimes \C }^n( P \boxtimes M, P \boxtimes M )$ must be nonzero for infinitely many $n \ge 1$. The Deligne product $\C \boxtimes \C$ satisfies \textbf{Fg} (again from the paragraph preceding the theorem), hence by using \cite[Proposition 6.2 and Corollary 4.2]{BPW1} again we see that $P \boxtimes M$ is not projective in $\C \boxtimes \C$. This implies that $X = ( P \boxtimes M, \alpha )$ is not projective in $\D$, as explained in the proof of Lemma \ref{lem:crossedproductfg}. Consequently, the support variety $V_{\D}(X)$ is not trivial, again by \cite[Corollary 4.2]{BPW1}.

Now consider the object $X \ot_2 X$. By definition of the tensor product in $\D$, we obtain
\begin{eqnarray*}
X \ot_2 X & = & \left ( ( P \boxtimes M ) \ot_1 \alpha_* ( P \boxtimes M ), \alpha^2 \right ) \\
& = & \left ( ( P \boxtimes M ) \ot_1 ( M \boxtimes P ), e \right ) \\
& = & \left ( (P \ot M) \boxtimes (M \ot P), e \right )
\end{eqnarray*}
where $e$ is the identity element of $C_2$. Let us denote the objects $P \ot M$ and $M \ot P$ in $\C$ by $Q_1$ and $Q_2$, respectively; these are both projective, since the projective objects form an ideal. As in the previous paragraph, there is an isomorphism
$$\Ext_{ \C \boxtimes \C }^* ( Q_1 \boxtimes Q_2, Q_1 \boxtimes Q_2 ) \simeq \Ext_{\C}^* (Q_1,Q_1) \ot_k \Ext_{\C}^* (Q_2,Q_2)$$
of $k$-vector spaces, and so since $Q_1$ and $Q_2$ are projective in $\C$, we conclude this time that $\Ext_{ \C \boxtimes \C }^* ( Q_1 \boxtimes Q_2, Q_1 \boxtimes Q_2 ) = 0$ for all $n \ge 1$. Therefore, by \cite[Proposition 6.2 and Corollary 4.2]{BPW1}, the object $Q_1 \boxtimes Q_2$ is projective in $\C \boxtimes \C$. Again, as explained in the proof of Lemma \ref{lem:crossedproductfg}, we now see that $X \ot_2 X = ( Q_1 \boxtimes Q_2, e )$ is projective in $\D$, hence the support variety $V_{\D}( X \ot_2 X)$ is trivial. This shows that $V_{\D} ( X \ot_2 X ) \neq V_{\D} ( X )$.
\end{proof}

In general, each factor in a Deligne tensor product embeds into it, with a structure preserving functor. Thus if $\C$ and $\D$ are finite tensor categories, then $\C$ embeds (as a finite tensor category) into $\C \boxtimes \D$ via $C \mapsto C \boxtimes \unit_{\D}$, and similarly for morphisms. Using this, we see that $\C$ embeds as a finite tensor category into $( \C \boxtimes \C )  \rtimes C_2$ via $C \mapsto (C \boxtimes \unit, e)$. Consequently, Theorem \ref{thm:main} shows that over a perfect field, \emph{any} finite tensor category that satisfies \textbf{Fg} embeds into one that also satisfies \textbf{Fg}, but not the tensor product property for support varieties---even when the tensor product property \emph{does} hold for the original category.

\begin{corollary}\label{cor:main}
Let $k$ be a perfect field and $( \C, \ot_{\C}, \unit_{\C} )$ a non-semisimple finite tensor $k$-category that satisfies \emph{\textbf{Fg}}. Then $( \C, \ot_{\C}, \unit_{\C} )$ embeds as a finite tensor category into a finite tensor $k$-category $( \D, \ot_{\D}, \unit_{\D} )$ that also satisfies \emph{\textbf{Fg}}, but not the tensor product property for support varieties.
\end{corollary}

\begin{remark}
The crossed product category $( \C \boxtimes \C )  \rtimes C_2$ from Theorem \ref{thm:main} is not braided. This can be seen directly from the proof, by involving Proposition \ref{prop:powers}: the object $X$ from the proof is not projective in $( \C \boxtimes \C )  \rtimes C_2$, whereas the tensor product $X \ot_2 X$ is (here $\ot_2$ denotes the tensor product in $( \C \boxtimes \C )  \rtimes C_2$, as in the proof). One can also convince oneself in a more direct way. Namely, let $M$ be an object in $C$, and denote by $\ot_1$ the tensor product in $\C \boxtimes \C$, again as in the proof of Theorem \ref{thm:main}. Then
\begin{eqnarray*}
( M \boxtimes \unit, \alpha ) \ot_2 ( \unit \boxtimes \unit, \alpha ) & = & \left ( (M \boxtimes \unit ) \ot_1 \alpha_* (  \unit \boxtimes \unit ), \alpha^2 \right ) \\
& = & \left ( (M \boxtimes \unit ) \ot_1 (  \unit \boxtimes \unit ), e \right ) \\
& = & \left ( M \boxtimes \unit, e \right )
\end{eqnarray*}
whereas
\begin{eqnarray*}
( \unit \boxtimes \unit, \alpha ) \ot_2 ( M \boxtimes \unit, \alpha ) & = & \left ( (\unit \boxtimes \unit ) \ot_1 \alpha_* (  M \boxtimes \unit ), \alpha^2 \right ) \\
& = & \left ( (\unit \boxtimes \unit ) \ot_1 (  \unit \boxtimes M ), e \right ) \\
& = & \left ( \unit \boxtimes M, e \right )
\end{eqnarray*}
The objects $( M \boxtimes \unit, e)$ and $( \unit \boxtimes M, e)$ are isomorphic in $( \C \boxtimes \C )  \rtimes C_2$ if and only if the objects $M \boxtimes \unit$ and $\unit \boxtimes M$ are isomorphic in $\C \boxtimes \C$. This is not the case in general.
\end{remark}

In light of the remark, we ask the following question.

\begin{question}
Does every braided finite tensor category that satisfies \textbf{Fg} also satisfy the tensor product property for support varieties?
\end{question}



\end{document}